\documentclass[12pt,oneside,a4paper]{amsart}
\pdfoutput=1
\usepackage{amssymb}
\usepackage{hyperref}

\usepackage{mathtools}
\mathtoolsset{showonlyrefs}

\usepackage[english]{babel}

\theoremstyle{plain}
\newtheorem{theorem}{Theorem}[section]
\newtheorem{lemma}[theorem]{Lemma}

\newtheorem{proposition}[theorem]{Proposition}
\theoremstyle{remark}
\newtheorem{remark}[theorem]{Remark}

\newcommand{\R}{\mathbb{R}}
\newcommand{\Z}{\mathbb{Z}}
\newcommand{\C}{\mathbb{C}}
\newcommand{\Q}{\mathbb{Q}}
\newcommand{\F}{\mathbb{F}}
\newcommand{\K}{\mathbb{K}}
\newcommand{\abs}[1]{\left| #1 \right|}
\DeclareMathOperator{\tr}{tr}
\DeclareMathOperator{\Aut}{Aut}
\newcommand{\rt}{R}
\newcommand{\mrt}{\bar R}

\title{On Radon transforms on finite groups}
\author{Joonas Ilmavirta}
\address{Department of Mathematics and Statistics, University of Jyv\"askyl\"a, P.O.Box 35 (MaD) FI-40014 University of Jyv\"askyl\"a, Finland}
\email{joonas.ilmavirta@jyu.fi}
\date{\today}

\begin{document}

%
%
%

\begin{abstract}
If $G$ is a finite group, is a function $f:G\to\mathbb C$ determined by its sums over all cosets of cyclic subgroups of $G$? In other words, is the Radon transform on $G$ injective? This inverse problem is a discrete analogue of asking whether a function on a compact Lie group is determined by its integrals over all geodesics. We discuss what makes this new discrete inverse problem analogous to well-studied inverse problems on manifolds and we also present some alternative definitions. We use representation theory to prove that the Radon transform fails to be injective precisely on Frobenius complements. We also give easy-to-check sufficient conditions for injectivity and noninjectivity for the Radon transform, including a complete answer for abelian groups and several examples for nonabelian ones.
\end{abstract}



\maketitle

\section{Introduction}

Before introducing the problem we wish to study, let us begin with an older problem.
Given a closed Riemannian manifold~$M$ and a smooth function $f:M\to\C$, can we recover the function~$f$ if we know its integral over all closed geodesics?
This problem was first considered by Funk~\cite{F:S2-radon} in 1913 in the case $M=S^2$ -- it turned out that one can only recover the even part of the function.
This problem has since received attention from several authors, and results have been proven on essentially two types of manifolds: those with algebraic structure (symmetric spaces, Grassmannians and Lie groups~\cite{I:torus,G:grassmann,G:lie-radon,GJ:radon-symm-space,I:lie-radon}) and those with negative curvature or some other similar structure (see for example~\cite{DS:anosov}).

This problem was recently considered on compact Lie groups~\cite{I:lie-radon} and the aim of the present paper is to extend that approach to finite groups.
We ask whether a function $f:G\to\C$ on a finite group~$G$ can be determined from its integrals over all closed geodesics.
In other words, we ask whether the Radon transform that sends a function to its integrals over all closed geodesics is injective.

The first problem is to make sense of the question.
The integral is, of course, a sum, but the definition of a geodesic is less obvious.
We define a closed geodesic to be a translate of a cyclic subgroup; as it turns out, this is in many ways analogous closed geodesics on a Lie group (see section~\ref{sec:comparison}).
A more careful formulation of the problem is given in section~\ref{sec:description}.


\subsection{Summary of results}
\label{sec:summary}

In the case of finite abelian groups, we give a complete answer to our problem.
The Radon transform is injective on a finite abelian group if and only if it is not cyclic (see theorem~\ref{thm:abelian}).
It follows, for example, that the Radon transform is not injective on any finite abelian group of squarefree order.

We also give a complete answer for product groups.
The Radon transform fails to be injective on a product if and only if the constituents have noninjective Radon transform and coprime order (see theorem~\ref{thm:prod}).

Our answers in other cases are far less complete using elementary methods.
We show that the Radon transform is injective on all symmetric, alternating and dihedral groups, but noninjective on all dicyclic groups.
We also give sufficient and necessary conditions for the Radon transform to be injective, but none of our conditions is both sufficient and necessary.
For these results, see section~\ref{sec:nonabelian}.

We give a representation theoretical characterization of the finite groups where the Radon transform is injective (see theorem~\ref{thm:rep-radon-inj}).
It turns out that the Radon transform fails to be injective on~$G$ if and only if~$G$ is a Frobenius complement.
(This is true whenever the functions on which the Radon transform operates take values in a field of characteristic zero, not necessarily complex numbers.)
This is equivalent with the existence of a fixed-point-free representation in a suitable sense; see theorem~\ref{thm:rep-radon-inj}.
Tools from representation theory also allow us to give a second proof of the result for abelian groups and a characterization of the kernel of the Radon transform when it is not injective (see theorem~\ref{thm:rep-radon-ker}).

Many of the other results follow from this characterization and known properties of Frobenius complements.
We however prefer to give more elementary proofs of our results -- in particular proofs written in the language of Radon transforms.
The characterization result is also a new characterization of Frobenius complements.
We will discuss this matter again in remark~\ref{rmk:frobenius} and under it.

For any given finite group it is a straightforward but tedious exercise to see if the Radon transform is injective (see section~\ref{sec:observations}).
We do not consider this a satisfactory solution to our problem.
Our results, besides providing some understanding of the problem at hand, also significantly speed up the aforementioned banal exercise.


We will also discuss some generalizations and variations of the problem in section~\ref{sec:other}.
We present some remarks about the Radon transform on infinite discrete groups in section~\ref{sec:Z} and we introduce the discrete geodesic flow in section~\ref{sec:flow}.
In section~\ref{sec:var} we will introduce a variant of the Radon transform where geodesics are assumed to satisfy an additional maximality condition and study its basic properties.


\section{Description of the problem}
\label{sec:description}

\subsection{Definition of geodesics}
\label{sec:def-geod}

Let~$G$ be a finite group.
For any integer $n\geq2$ let~$C_n$ denote the cyclic group of order~$n$.
A geodesic of length~$n$ on the group~$G$ is a mapping $C_n\ni t\mapsto x\gamma(t)\in G$, where $x\in G$ and $\gamma:C_n\to G$ is a nontrivial homomorphism.
The variable $t\in C_n$ is the discrete analogue of time (or arc length).
We will often ignore length and refer only to geodesics, by which we mean geodesics of all possible lengths.
We denote by~$\Gamma_n$ the set of all nontrivial homomorphisms $C_n\to G$.

Note that the definition of a geodesic is symmetric between left and right multiplication, although the definition is given only with left translations (by the element~$x$).
If~$\gamma$ is a nontrivial homomorphism, so is $t\mapsto x^{-1}\gamma(t)x$.

It is common to identify a curve with its image.
Similarly, geodesics in our setting can be thought of as mappings to~$G$ or as subsets of~$G$.
With this identification we could redefine geodesics to be cosets of cyclic subgroups that contain more than one point.
Note that the length of a geodesic and the number of points in its image may not coincide, but in this case each point in the image has equally large preimage.

We will show that only geodesics whose length is a prime number matter.
In this case identifying the mapping and the point is safe, since a homomorphism $C_p\to G$ is nontrivial if and only if it is injective if~$p$ is a prime.
Also, any subgroup of prime order is automatically cyclic, which would allow for an even shorted definition, but at the cost of diminished clarity.

If the group~$G$ were infinite, there could also be infinitely long geodesics, corresponding to homomorphisms $\Z\to G$.
We restrict ourselves to finite groups for simplicity.
The Radon transform in infinite discrete groups is discussed briefly in section~\ref{sec:Z}, and it turns out that our inverse problem can behave very differently on infinite groups.

Actually, on finite groups every geodesic corresponds to a homomorphism from~$\Z$, as such a homomorphism must factor through some finite cyclic group.
(Every geodesic on a finite group is necessarily periodic.)
But it is technically convenient to restrict attention to homomorphisms from finite cyclic groups.

This is not the only reasonable definition of a geodesic in a finite group.
For an alternative, see section~\ref{sec:var}.

Geodesics on finite groups can also be thought of as trajectories of a discrete time dynamical system.
See section~\ref{sec:flow} for this aspect.

\subsection{Comparison of discrete and continuous geodesics}
\label{sec:comparison}

There are several reasons to see these as natural generalizations of geodesics on Lie groups (with a bi-invariant metric).
From now on, we assume all Lie groups to have dimension one or higher, so that finite groups are not Lie groups.
This is not essential, but makes it easier to distinguish the two cases.
By geodesics on Lie groups we mean only the periodic ones here.

For every nontrivial group there exist geodesics.
---
The same is true for compact Lie groups of nonzero dimension.

Geodesics remain geodesics under left and right translations. Left translations were built into the definition, and right invariance is due to the fact that if $\gamma:C_n\to G$ is a homomorphism, so is the map $\gamma_x:C_n\to G$ with $\gamma_x(t)=x^{-1}\gamma(t)x$.
---
The same is true for compact Lie groups.

If the finite group~$G$ is abelian, it is a product of cyclic groups.
---
If a connected, compact Lie group is abelian, it is a product of copies of~$S^1$ (a torus, that is).

Subgroups are totally geodesic; that is, if $\gamma:C_n\to H$ is a geodesic on a subgroup $H<G$, then it is also a geodesic on~$G$.
---
The same is true for compact Lie groups.

The groups~$C_n$ are abelian and once a generator has been fixed, there is a natural meaning to going forward or backward, and there are no other directions.
---
A similar description fits~$S^1$.

A geodesic is uniquely determined by giving two adjacent points\footnote{Two elements~$t$ and~$s$ of a cyclic group can be naturally considered adjacent if $ts^{-1}$ generates the group.} on it.
---
A geodesic on a Lie group is uniquely determined by a point and a direction (although it need not be closed).

The one property of geodesics that our generalization does not capture is minimization of length.
This was only expected, as distance is much less of a meaningful concept on a finite group.
There are studies of norms on finite groups~\cite{B:finite-group-norm}, but the author is unaware of a pre-existing theory of geodesics.

\subsection{The Radon transform on finite groups}

Let~$G$ again be a finite group and let~$A^\C$ denote the set of complex valued functions on a set~$A$.
Recall that~$\Gamma_n$ denotes the set of nontrivial homomorphisms $C_n\to G$.
For $f\in G^\C$ and $\gamma\in\Gamma_n$, $n\geq2$, we denote
\begin{equation}
\rt_nf(x,\varphi)
=
\sum_{t\in C_n}f(x\gamma(t)).
\end{equation}
This is the sum of~$f$ over the geodesic defined by the homomorphism~$\gamma$ and shifted by~$x$.
The mapping $\rt_n:G^\C\to (G\times\Gamma_n)^\C$ is analogous to the Radon transform on Lie groups.
Since it is a discrete analogue of an integral transform, we find it natural to describe it as a sum transform.

A similar definition of the Radon transform on Lie groups was given in~\cite{I:lie-radon}.
To make the similarity to the classical Radon transform on~$\R^n$ even more transparent, we remark that every straight line in~$\R^n$ is a translate of a nontrivial Lie homomorphism $\R\to\R^n$.

We want to consider geodesics of all lengths at once, so we will denote by~$\rt f$ the sequence $(\rt_nf)_{n=2}^\infty$ of all of the sum transforms~$\rt_n$.
Therefore when we say that~$\rt$ is injective, we mean that~$f$ can be recovered if we know~$\rt_nf$ for all~$n$.
By the kernel of~$\rt$ we mean $\ker\rt=\bigcap_{n=1}^\infty\ker\rt_n\subset\C^G$, and injectivity is equivalent with $\ker\rt=0$.

We call~$\rt$ the (discrete) Radon transform on the group~$G$.
Our main problem is to find out when this transform is injective.

By definition~$\rt$ consists of infinitely many sum transforms, but as we will see in lemma~\ref{lma:prime}, it suffices to consider only those~$n$ that are prime factors of~$\abs{G}$.

It is no way crucial that we assume the functions to take values in~$\C$.
Our results remain intact when~$\C$ is replaced with any vector space over a field of characteristic zero.
Working over the complex numbers becomes technically very useful in section~\ref{sec:rep}, where we use representation theory.

It is obvious that if the Radon transform is injective over a field, then it is also injective over any subfield.
The converse is also true, but less obvious.
Let~$\F$ be a subfield of~$\K$ and let~$\rt^\F$ and~$\rt^\K$ be the corresponding Radon transforms on~$G$.
We can regard~$\K$ as a vector space over~$\F$, and we can use its dual space~$\K^*$ to deduce injectivity over~$\K$ from injectivity over~$\F$.
If $f\in\K^G$ and $u\in\K^*$, we have $u(\rt^\K f)=\rt^\F u(f)$.
If~$\rt^\F$ is injective, then $\rt^\K f=0$ implies that $u(f)=0$ for all $u\in\K^*$ and therefore that $f=0$.

This argument for field extensions also works with other vector spaces.
If we want to study the Radon transforms of functions taking values in a vector space, it suffices to consider the problem with functions that take values in the underlying field or any subfield.
Therefore our results (which are stated over~$\C$) are true over any vector space over any field of characteristic zero.
The Radon transform itself makes sense if the functions take values in any abelian group.


The classical Radon transform has also been considered on noncompact manifolds.
The Radon transform in Euclidean spaces has applications in medical and other imaging and has been widely studied.
In the Radon transform one typically integrates over hyperplanes and in the X-ray transform over lines, but we refer to all integral transforms of this kind as Radon transforms for simplicity.
For an introduction to the classical Radon transform, see~\cite{book-helgason,book-natterer}.

\subsection{Comparison of discrete and continuous Radon transforms}

In section~\ref{sec:comparison} we compared geodesics in the discrete and continuous settings and highlighted the similarities between the two theories.

There are also differences between the continuous and discrete theories, and this is reflected in the differences in injectivity results for the Radon transform.
The analogous study of the Radon transform on Lie groups was based on finding toric subgroups and reducing the problem to the abelian case.
There are only four compact, connected Lie groups that do not have a two dimensional torus as a subgroup: the trivial group,~$SO(2)$, $SO(3)$ and~$SU(2)$.

The analogous approach does not work with finite groups.
The discrete analogue of an at least two dimensional toric subgroup is a non-cyclic abelian subgroup.
Unfortunately there is no simple description of groups with such subgroups.
In particular there are infinitely many groups without such subgroups.

The most immediate analogue of a maximal torus would be a maximal abelian subgroup, but two such subgroups are not necessarily conjugate.
In this sense a better analogue is given by Sylow $p$\mbox{-}subgroups, since any two Sylow $p$\mbox{-}subgroups are conjugate just like any two maximal tori, but Sylow $p$\mbox{-}subgroups need not be abelian.
The existence of convenient subgroups is a much more complicated issue in finite groups.

Another difference lies in the internal symmetries of geodesics.
As mentioned in section~\ref{sec:comparison}, a discrete geodesic has a well defined direction once a generator is chosen in the cyclic group.
There are many possible choices, corresponding to automorphisms of the cyclic group.
There is no natural generator for the cyclic subgroup corresponding to a geodesic, so there is no ``natural direction'' of discrete geodesics.
The group~$S^1$ that underlies continuous geodesics has only two automorphisms -- the nontrivial one is the reversal -- and the only ambiguity is in choosing one of the two directions.

These structural differences between discrete and continuous geometry of groups leads us to a different course of action than in~\cite{I:lie-radon}.

\subsection{Some simple observations}
\label{sec:observations}

Let us begin our enquiry by presenting some first observations.

Given any single finite group~$G$ it is a straightforward task to see whether the Radon transform is injective or not.
There are finitely many points and geodesics on~$G$, so the Radon transform is a linear operator between finite dimensional spaces.
Checking injectivity therefore reduces to forming a matrix and calculating its rank.


Even in more general situations, dimension counting arguments suffice for proving that the Radon transform has a nontrivial kernel.
More elaborate arguments are needed to prove that the transform is injective, if that is the case.

Suppose we know the sum of $f:G\to\C$ over all geodesics and want to reconstruct~$f$.
It suffices to reconstruct~$f$ at~$e$.
Indeed, let $\tau_xg:G\to\C$, $\tau_xf(y)=f(xy)$, be the left translation of~$f$ by $x\in G$ and let~$S$ be the assumedly existing solution operator that takes $\rt f\mapsto f(e)$.
Since geodesics are translation invariant, we know $\rt\tau_xf$ as well.
Then $S(\rt\tau_xf)=f(x)$.

Besides sums over geodesics, we shall also be occasionally interested in sums over the whole group.
If $\gamma:C_n\to G$ is a homomorphism, then
\begin{equation}
\label{eq:group-sum}
\sum_{x\in G}f(x)
=
\frac{1}{n}\sum_{x\in G}\rt_nf(x,\gamma).
\end{equation}
The Radon transform thus determines the average of the function.

The following lemma demonstrates that we only need to consider geodesics of prime length.
Therefore, if we know all subgroups of prime order, we immediately know all relevant geodesics.

\begin{lemma}
\label{lma:prime}
Let~$G$ be a finite group.
The Radon transforms~$\rt_nf$ of $f:G\to\C$ is uniquely determined for each $n\geq2$ by the knowledge of~$\rt_pf$ for every prime~$p$.
\end{lemma}

\begin{proof}
Consider the Radon transform~$\rt_nf$ for some composite number $n=km$.
It suffices to prove that~$\rt_nf$ is determined by~$\rt_kf$ and~$\rt_mf$.
Take any nontrivial homomorphism $\gamma:C_n\to G$.
Let $\alpha:C_m\to C_n$ be the embedding that multiplies by~$k$ (where we have identified $C_m=\Z/m\Z$ and similarly for~$C_n$ and~$C_k$).

Now $\gamma\circ\alpha:C_m\to G$ is a homomorphism.
If it is trivial, then $\gamma(k)=e$ and we may define a homomorphism $\beta:C_k\to G$ by letting $\beta(t)=\gamma(t)$.
Then $\rt_nf(x,\gamma)=m\rt_kf(x,\beta)$.

If $\gamma\circ\alpha$ is not trivial, then $\rt_nf(x,\gamma)=\frac{1}{m}\sum_{t\in C_n}\rt_mf(x\gamma(t),\gamma\circ\alpha)$.
In either case $\rt_nf(x,\gamma)$ is determined by~$\rt_kf$ and~$\rt_mf$.
\end{proof}

\begin{lemma}
\label{lma:subgroup}
Let~$G$ be a finite group.
If the Radon transform is injective on any subgroup $H<G$, then it is injective on~$G$ as well.
\end{lemma}

\begin{proof}
Let $f:G\to\C$ be a function which we wish to recover from its Radon transform.
If we restrict our attention to geodesics that lie in~$H$, we obtain the Radon transform of~$f|_H$ on~$H$.
This is transform injective, so it determines~$f(e)$, whence~$f$ can be recovered everywhere in~$G$.
\end{proof}

We will repeatedly use the following special case of lemma~\ref{lma:subgroup}, so we state it separately.

\begin{lemma}
\label{lma:product}
Suppose $G=G_1\times G_2$, where each~$G_i$ is a finite abelian group.
If~$\rt$ is injective on~$G_1$, it is injective on~$G$.
\end{lemma}


\section{Abelian groups}
\label{sec:abelian}

We consider first the case where the finite group~$G$ is abelian.
For the sake of a simple presentation, we will only present our lemmas in the generality we need here.
By Kronecker's decomposition theorem any finite abelian group~$G$ can be written as a product $G=C_{p_1^{k_1}}\times\cdots\times C_{p_N^{k_N}}$ for some~$N$, where each~$p_i$ is a prime and each~$k_i$ a natural number.

\begin{lemma}
\label{lma:product-geodesic}
Suppose $G=C_{p_1^{k_1}}\times\cdots\times C_{p_N^{k_N}}$ and $\gamma:C_p\to G$ is a homomorphism.
If $\pi_i:G\to C_{p_i^{k_i}}$ is the projection, the homomorphism~$\pi_i\circ\gamma$ is trivial unless $p=p_i$.
\end{lemma}

\begin{proof}
This proof follows from the observation that there are no nontrivial homomorphisms $C_p\to C_{q^k}$ for distinct primes~$p$ and~$q$.
\end{proof}

\begin{lemma}
\label{lma:prime-power}
If $G=C_{p^k}$ for~$p$ prime, then~$\rt$ is not injective.
\end{lemma}

\begin{proof}
By lemma~\ref{lma:prime} it suffices to consider geodesics of prime length, and thus only geodesics of length~$p$.
There are~$p^{k-1}$ such geodesics (identifying geodesics that differ by an automorphism of~$C_{p^i}$), so~$\rt_pf$ only gives~$p^{k-1}$ numbers.
This is not enough to determine a function on the group~$G$ for dimensional reasons.
\end{proof}

\begin{lemma}
\label{lma:cyclic}
If~$G$ is cyclic, then~$\rt$ is not injective.
\end{lemma}

\begin{proof}
We use again the decomposition $G=C_{p_1^{k_1}}\times\cdots\times C_{p_N^{k_N}}$.
If $N=1$, the claim is given by lemma~\ref{lma:prime-power}.
We assume thus that $N\geq2$.

The primes $p_1,\dots,p_N$ are distinct, since otherwise~$G$ would not be cyclic.
We only need to consider geodesics of these prime lengths.

For any~$i$ let $f_i:C_{p_i^{k_i}}\to\C$ be a nontrivial function in the kernel of~$\rt_{p_i}$, whose existence is guaranteed by lemma~\ref{lma:prime-power}.
Define $f:G\to\C$ by letting $f(x_1,\dots,x_N)=f_1(x_1)\cdots f_N(x_N)$.
It remains to prove that $f\in\ker\rt$.

Indeed, take any index~$i$ and consider a homomorphism $\gamma\in\Gamma_{p_i}$ and $x\in G$.
By lemma~\ref{lma:product-geodesic} there is a homomorphism $\varphi:C_{p_i}\to C_{p_i^{k_i}}$ such that $\gamma(t)=(e,\dots,e,\varphi(t),e,\dots,e)$ where~$\varphi$ occurs in the~$i$th position.
Now
\begin{equation}
\rt_{p_i}f(x,\gamma)
=
\rt_{p_i}f_i(x_i,\varphi)
\prod_{\mathclap{n\in\{1,\dots,N\}\setminus\{i\}}}f_n(x_n)
=
0.
\qedhere
\end{equation}
\end{proof}

\begin{lemma}
\label{lma:square}
If $G=C_p\times C_p$ for a prime~$p$, then~$\rt_p$ is injective.
\end{lemma}

\begin{proof}
Each point in~$G$ corresponds to a unique homomorphism $C_p\to G$, so $\#\Gamma_p=p^2-1$.
Each $x\in G\setminus\{e\}$ belongs to the image of $p-1$ homomorphisms $C_p\to G$, since the automorphism group of~$C_p$ has order $p-1$.
Thus
\begin{equation}
\label{eq:vv2}
\sum_{\gamma\in\Gamma_p}\rt_pf(x,\gamma)
=
(p^2-p)f(x)+(p-1)\sum_{y\in G}f(y)
\end{equation}
for every $x\in G$.

If $\rt_pf=0$, the identities~\eqref{eq:group-sum} and~\eqref{eq:vv2} give $f(x)=0$ for every $x\in G$.
Thus~$\ker\rt_p$ is trivial.
\end{proof}

In fact, the proof gives a reconstruction formula in the case $G=C_p\times C_p$, valid for any $x\in G$ and $\varphi\in\Gamma_p$:
\begin{equation}
f(x)
=
\frac{1}{p^2-p}\sum_{\gamma\in\Gamma_p}\rt_pf(x,\gamma)
-
\frac{1}{p^2}\sum_{y\in G}\rt_pf(y,\varphi).
\end{equation}
Compare this with the normal operator of the X-ray or Radon transform in the Euclidean case\footnote{If~$\rt^*$ is the adjoint of~$\rt$, the normal operator~$\rt^*\rt$ of~$\rt$ is such that~$\rt^*\rt f(x)$ is the integral over the integrals over all lines passing through~$x$. The normal operator can be used to invert the Radon transform~\cite{book-helgason}.}.

\begin{theorem}
\label{thm:abelian}
Let~$G$ be a nontrivial finite abelian group.
Then the following are equivalent:
\begin{enumerate}
\item $G$ is not cyclic.
\item $G$ contains a square of a cyclic group.
\item $\rt$ is injective on~$G$.
\end{enumerate}
\end{theorem}

\begin{proof}
Equivalence of the first two conditions follows from the fact that~$G$ is a product of cyclic groups.
In fact, both of these conditions are equivalent with the statement that at least one prime occurs at least twice in the decomposition $G=C_{p_1^{k_1}}\times\cdots\times C_{p_N^{k_N}}$.
Lemma~\ref{lma:cyclic} shows that the third condition implies the first one.
Lemmas~\ref{lma:square} and~\ref{lma:product} show that the second condition implies the third one.
\end{proof}

\section{Product groups}
\label{sec:prod}

\subsection{Direct products}

We now turn to the question of determining injectivity of the Radon transform on a product group if the corresponding results are known on the constituent groups.
We are able to give a rather complete answer.

\begin{theorem}
\label{thm:prod}
Let~$G_1$ and~$G_2$ be two finite groups.
Then the Radon transform is not injective on $G_1\times G_2$ if and only if it is not injective on either of the constituent groups and the orders of the constituent groups are coprime.
\end{theorem}

\begin{proof}
We need to show the following:
\begin{enumerate}
\item If the Radon transform is injective on either constituent group, then it is injective on the product.
\item If the orders of the constituent groups have a common prime factor, then the Radon transform is injective on the product.
\item If the orders of the constituent groups are coprime and the Radon transform is noninjective on both, then it is noninjective on the product.
\end{enumerate}

1. This is the statement of lemma~\ref{lma:product}.

2.
Let~$p$ be a prime that divides~$\abs{G_i}$ for both $i=1,2$.
By the existence of Sylow subgroups each~$G_i$ has a subgroup isomorphic to the cyclic group~$C_p$.
Thus the product $G_1\times G_2$ has a subgroup isomorphic to $C_p\times C_p$.
By lemma~\ref{lma:subgroup} and theorem~\ref{thm:abelian} the Radon transform is injective on $G_1\times G_2$.

3.
It suffices to consider geodesics of prime length.
Let~$p$ be a prime and~$\gamma:C_p\to G_1\times G_2$ a nontrivial homomorphism.
By assumption $p$ can only divide the order of one of the two groups~$G_1$ and~$G_2$, so~$\gamma$ composed with one of the projections must be trivial.
Thus if~$f_i$ is in the kernel of the Radon transform on~$G_i$, the function $f(x_1,x_2)=f_1(x_1)f_2(x_2)$ is in the kernel of the Radon transform on~$G$.
\end{proof}

\subsection{Semidirect products}

The above result about direct products does not generalize easily to semidirect products.
Let us attempt to describe this issue.

Let~$N$ and~$H$ be finite groups and $\varphi:H\to\Aut(N)$ a homomorphism, and consider the semidirect product $G=N\rtimes_\varphi H$.
We consider~$N$ and~$H$ as subgroups of~$G$.
We want to study cyclic subgroups of~$G$ if~$\abs{N}$ are~$\abs{H}$ coprime; such study was at the heart of the proof of theorem~\ref{thm:prod}.

Let~$p$ be a prime that divides~$\abs{N}$ but not~$\abs{H}$ and consider an element $nh\in G$ of order~$p$, where $n\in N$ and $h\in H$.
Now $(nh)^p=\tilde n\tilde h$ with $\tilde n=n\varphi_{h}(n)\varphi_{h^2}(n)\cdots\varphi_{h^{p-1}}(n)\in N$ and $\tilde h=h^p\in H$.
The condition $(nh)^p=e_G$ implies that $\tilde n=e_N$ and $\tilde h=e_H$, the latter of which implies that $h=e_H$ since~$p$ does not divide~$\abs{H}$.
Thus $\langle nh\rangle\leq N<G$.

If, instead,~$p$ divides~$\abs{H}$ but not~$\abs{N}$, we \emph{cannot} draw the conclusion that $\langle nh\rangle\leq H<G$ without assuming anything much short of~$\varphi$ being trivial and hence the product being direct.
The problem is that~$\tilde n$ and~$\tilde h$ can both be nontrivial even though $(nh)^p=e_G$.
For example, let~$q$ be a prime distinct from~$q$ and let $H=C_p$ and let~$N$ be the product of~$p$ copies of~$C_q$ indexed by~$C_p$.
Let~$H$ act on~$N$ by permuting the indices.
For any $h\in H\setminus\{e_H\}$ and $n=(a_1,\dots,a_p)\in C_q^p=N$ such that $a_1a_2\cdots a_p=e_{C_q}$ a simple calculation shows that $\tilde n=e_N$ and $\tilde h=e_H$.

To further demonstrate problems with semidirect products, let us mention two examples.
All dihedral are semidirect products of cyclic groups.
The Radon transform fails to be injective in all cyclic groups, but is is injective on all dihedral groups.
On the other hand, it is never injective on a dicyclic group, and~$Dic_6$ is a semidirect product of~$Dic_2$ and~$C_3$.
It is not clear whether there is a simple version of theorem~\ref{thm:prod} for semidirect products.
These and other examples are studied in the next section.

\section{Nonabelian groups}
\label{sec:nonabelian}

Combining theorem~\ref{thm:abelian} with lemma~\ref{lma:subgroup} we see that the Radon transform is injective on~$G$ if~$G$ contains an abelian subgroup that is not cyclic.
This condition is simple, but it does not cover all nonabelian groups.

A natural construction whose relation to injectivity of the Radon transform one may consider is the quotient.
It does not hold that if the Radon transform is injective on~$G$, it would also be injective on any quotient~$G/H$.
The argument generalizing the one for Lie groups (see~\cite[Proposition~2.4]{I:lie-radon}) fails at a critical point.

Let us present it here:
Let~$G$ be a finite group,~$H$ its normal subgroup, $\pi:G\to G/H$ the projection, $\gamma:C_p\to G$ a geodesic and $f:G/H\to\C$ a function.
Indicating the underlying group by a superscript in the transform, we get $\rt_p^G(\pi^*f)(x,\gamma)=\rt^{G/H}_pf(\pi(x),\pi\circ\gamma)$.
It seems that if~$\rt_p^G$ is injective, then we could recover~$\pi^*f$ and thus~$f$ from the knowledge of~$\rt^{G/H}_pf$.
The reason why this fails is that $\pi\circ\gamma$ may be trivial and hence may not correspond to a geodesic.
If we allowed trivial homomorphisms in our definition of geodesics, every Radon transform on every (finite or infinite) group would be trivially injective.

Let us give a more specific example.
Let~$H$ be a finite group with an injective Radon transform and let $n\geq2$ be an integer.
The Radon transform is injective on $G=H\times C_n$ by lemma~\ref{lma:product}.
But $H\times\{e\}$ is a normal subgroup of~$G$ and the quotient equals~$C_n$, but the Radon transform is not injective on~$C_n$ because of lemma~\ref{lma:cyclic}.

Also, there is an example where the Radon transform is not injective on~$G$, but it is injective on the quotient~$G/H$.
Indeed, if $G=Q_8$ is the quaternion group and~$H$ its only two element subgroup, then $G/H=C_2\times C_2$.
The Radon transform is not injective on~$G$ ($Q_8=Dic_2$; see proposition~\ref{prop:Dic_n}) but is injective on $C_2\times C_2$.

Proposition~\ref{prop:quotient} does give a result involving quotient groups.
It involves a variant of the Radon transform which we call the maximal Radon transform, but we refrain from discussing it before we introduce it in section~\ref{sec:var} in order to avoid confusion.

\subsection{Groups with an injective Radon transform}

Let us give explicit examples of families of finite groups where the Radon transform is injective.

\begin{proposition}
\label{prop:D_n}
The Radon transform is injective on the dihedral group~$D_n$ if and only if $n\geq3$.
\end{proposition}

\begin{proof}
The groups~$D_1$ and~$D_2$ are cyclic, so noninjectivity follows from theorem~\ref{thm:abelian}.
Suppose thus that $n\geq3$.

Let $f:D_n\to\C$ be a function that we wish to reconstruct from its Radon transforms.
By translation invariance of the problem it suffices to reconstruct~$f(e)$.

We may classify elements of~$D_n$ as reflections and rotations.
Let $\gamma_1,\dots,\gamma_n$ be the homomorphisms $C_2\to D_n$ going through the reflections.
There is a single homomorphism $\varphi:C_n\to D_n$ passing through all rotations.
Let $r\in D_n$ be any reflection, so that the geodesic $C_n\ni t\mapsto r\varphi(t)\in D_n$ passes through all reflections.
Now it is easy to observe that
\begin{equation}
\sum_{i=1}^n\rt_2f(e,\gamma_i)
-
\rt_nf(r,\varphi)
=
nf(e).
\qedhere
\end{equation}
\end{proof}

\begin{proposition}
\label{prop:S_n}
The Radon transform is injective on the symmetric group~$S_n$ if and only if $n\geq3$.
\end{proposition}

\begin{proof}
This follows from the previous theorem and lemma~\ref{lma:subgroup}, since $S_n=D_n$ for $n\leq3$ and~$S_n$ contains~$D_3$ as a subgroup if $n\geq4$.
\end{proof}

\begin{proposition}
\label{prop:A_n}
The Radon transform is injective on the alternating group~$A_n$ if and only if $n\geq4$.
\end{proposition}

\begin{proof}
The groups~$A_2$ and~$A_3$ are cyclic, so the Radon transform is not injective on them by theorem~\ref{thm:abelian}.
The group~$A_4$ contains $C_2\times C_2$ as a subgroup and~$A_4$ is a subgroup of~$A_n$ for $n\geq5$, so the claim follows from lemmas~\ref{lma:subgroup} and~\ref{lma:square}.
\end{proof}

\subsection{Groups with a noninjective Radon transform}

Let us start with a more explicit example and then present a more general condition for recognizing groups where the Radon transform is not injective.

\begin{proposition}
\label{prop:Dic_n}
The Radon transform is not injective on any dicyclic group~$Dic_n$, $n\geq2$.
\end{proposition}

\begin{proof}
The dicyclic group~$Dic_n$ is generated by two elements~$a$ and~$b$ that satisfy $a^{2n}=e$, $b^2=a^n$ and $ab=ba^{-1}$.
The elements of~$Dic_n$ can be written uniquely as~$b^ma^k$ with $0\leq m\leq1$ and $0\leq k\leq 2n-1$, and so the group has order~$4n$.

Let~$A$ be the cyclic subgroup spanned by~$a$.
This subgroup is isomorphic to~$C_{2n}$.
All cyclic subgroups of~$Dic_n$ that have order other than four are contained in~$A$; it is easy to check that every element in $Dic_n\setminus A$ has order four.
Therefore all geodesics of prime order are contained in~$A$ or its complement.

Now let $f:C_{2n}\to\C$ be a nontrivial function in the kernel of the Radon transform (such functions exist by theorem~\ref{thm:abelian}).
We may naturally regard~$f$ as a function on~$A$ and extend it by zero to~$Dic_n$.
The resulting function is in the kernel of the Radon transform on~$Dic_n$.
\end{proof}

\begin{proposition}
\label{prop:cyclic-subgroups}
Let~$G$ be a finite group and let~$S$ denote the set of its cyclic subgroups of prime order.
Then the Radon transform is not injective on~$G$ if
\begin{equation}
\label{eq:cyclic-subgroups}
\sum_{H\in S}\abs{H}^{-1}
<
1+\frac{\abs{S}-1}{\abs{G}}.
\end{equation}
\end{proposition}

\begin{proof}
It suffices to consider geodesics of prime order by lemma~\ref{lma:prime}.
Let~$p$ be a prime and let $\gamma:C_p\to G$ be a nontrivial homomorphism.
The image~$\gamma(C_p)$ is a cyclic subgroup of order~$p$.
The group~$G$ is composed of~$\abs{G}/p$ left translates of~$\gamma(C_p)$; these are precisely the geodesics corresponding to the homomorphism~$\gamma$.

Let $f:G\to\C$ be any function.
Summing over the Radon transform of~$f$ on these geodesics gives the sum of~$f$ over the whole~$G$; see~\eqref{eq:group-sum}.
The number of independent numbers contained in the Radon transform of~$f$ is therefore at most $1+\sum_{H\in S}(\abs{G}/\abs{H}-1)$, where the~$\pm1$s correspond to the sum of~$f$ over the whole group.
If this number is strictly less than~$\abs{G}$, the Radon transform has nontrivial kernel for dimensional reasons.
\end{proof}

Groups satisfying the condition of proposition~\ref{prop:cyclic-subgroups} include all cyclic groups and the dicyclic groups~$Dic_n$ for $2\leq n\leq14$.
The smallest dicyclic group not satisfying~\eqref{eq:cyclic-subgroups} is~$Dic_{15}$ with an inequality, and the smallest one with the opposite inequality is~$Dic_{30}$.
Taking~$n$ to be a primorial (a product of distinct primes less than a given number), we see that the inequality~\eqref{eq:cyclic-subgroups} can be violated as badly as we wish.
Nevertheless, the Radon transform remains noninjective by proposition~\ref{prop:Dic_n}, which demonstrates that the condition of proposition~\ref{prop:cyclic-subgroups} is far from sharp.

\section{A representation theoretical approach}
\label{sec:rep}

\subsection{Representations on finite groups}

As before, let~$G$ be a finite group.
If~$V$ is a vector space, a homomorphism $\rho:G\to GL(V)$ is called a representation of the group~$G$ and~$V$ is called the representation space of~$\rho$.
The underlying field of the vector space~$V$ could be anything, but we restrict our attention to the complex numbers.
Representation theory is particularly convenient over this field.

We will only consider finite dimensional representations, and we refer to the dimension of the space~$V$ as the dimension~$\dim(\rho)$ of the representation~$\rho$.
The space~$V$ can be equipped with an inner product so that~$\rho(x)$ is unitary for all $x\in G$.
We shall identify~$V$ with~$\C^{\dim(\rho)}$ and assume that~$\rho$ is unitary with respect to the usual inner product.

Two representations~$\rho$ and~$\sigma$ are equivalent if they have the same dimension and there is a unitary change of basis~$U$ so that $U\rho(x)U^{-1}=\sigma(x)$ for all $x\in G$.
The representation~$\rho$ reducible if there is a nontrivial splitting so that $V=V_1\oplus V_2$ and $\rho(x)V_i=V_i$ for all $x\in G$ and $i=1,2$.
A representation is faithful if it is an injective homomorphism.

We denote by~$\hat G$ the set of all irreducible representations of the group~$G$, including only one representation from each equivalence class.
The representations in~$\hat G$ can be reconstructed by decomposing the regular representation of~$G$.

For basics of representation theory of finite groups we refer to~\cite{S:rep-finite-grp}.
The representation theoretical approach presented here is similar to the one presented in~\cite[Section~3]{I:lie-radon} in the context of Lie groups.

\subsection{Fourier analysis on finite groups}

We wish to define the Fourier transform on finite groups analogously to the definition on Lie groups~\cite{I:lie-radon,RT:psdo-symm}.

We denote $\widehat{\C^G}=\bigoplus_{\rho\in\hat G}\C^{(\dim(\rho))^2}$.
The space~$\C^{(\dim(\rho))^2}$ is to be understood as the space of square matrices operating on the representation space of~$\rho$.
An alternative, less coordinate dependent option would be to define~$\widehat{\C^G}$ as the direct sum of the spaces of endomorphism of the representation spaces of the representations in~$\hat G$.

We define the Fourier transform as map $\C^G\to\widehat{\C^G}$, $f\mapsto\tilde f$, given by
\begin{equation}
\tilde f(\rho)
=
\sum_{x\in G}f(x)\rho(x)
\end{equation}
for $f\in\C^G$ and~$\rho$.
Note that~$\tilde f(\rho)$ is indeed an endomorphism of the representation space of~$\rho$.
We will denote the inner product of two matrices by $A:B=\tr(A^TB)$.

If the group~$G$ is cyclic, the Fourier transform defined above is just the usual discrete Fourier transform.
For functions depending on several variables only one of which ranges in~$G$, the Fourier transform is taken with respect to that variable.

\begin{lemma}
\label{lma:char-sum}
For $x\in G$ we have
\begin{equation}
\sum_{\rho\in\hat G}\dim(\rho)\tr(\rho(x))
=
\begin{cases}
\abs{G}, & x=e \\
0, & x\neq e.
\end{cases}
\end{equation}
\end{lemma}

\begin{proof}
The mapping $x\mapsto\tr(\rho(x))$ is called the character of the representation~$\rho$.
One of the standard orthogonality results for characters (see eg.~\cite[Section~2.3]{S:rep-finite-grp}) states that
\begin{equation}
\sum_{\rho\in\hat G}\overline{\tr(\rho(y))}\tr(\rho(x))
=
\begin{cases}
\abs{G}/c(x), & x\text{ and }y\text{ conjugate} \\
0, & \text{otherwise},
\end{cases}
\end{equation}
where~$c(x)$ is the size of the conjugacy class of~$x$ in~$G$.
Choosing $y=e$ and observing that $\tr(\rho(e))=\dim(\rho)$ gives the desired result.
\end{proof}

\begin{lemma}
\label{lma:fourier-bij}
The Fourier transform is a linear bijection.
\end{lemma}

\begin{proof}
We will actually prove that the Fourier transform is an isometry when~$\widehat{\C^G}$ is equipped with a suitable inner product.
This suffices since the Fourier transform is a linear mapping between two spaces of the same (finite) dimension.

Let $f,g\in\C^G$.
Using lemma~\ref{lma:char-sum} and $\rho(x)^*:\rho(y)=\tr(\rho(x^{-1}y))$ we have
\begin{equation}
\begin{split}
&\frac{1}{\abs{G}}\sum_{\rho\in\hat G}\dim(\rho)\tilde f(\rho)^*:\tilde f(\rho)
\\&\quad=
\sum_{x,y\in G}\overline{f(x)}f(y)
\frac{1}{\abs{G}}\sum_{\rho\in\hat G}\dim(\rho)\rho(x)^*:\rho(y)
\\&\quad=
\sum_{x\in G}\overline{f(x)}f(x).
\qedhere
\end{split}
\end{equation}
\end{proof}

\subsection{The Radon transform}

For a representation $\rho\in\hat G$ and a homomorphism $\gamma\in\Gamma_n$ we denote
\begin{equation}
I(\rho,\gamma)
=
\sum_{t\in C_n}\rho(\gamma(t)).
\end{equation}
The matrix $I(\rho,\gamma)$ is simply the sum of the representation~$\rho$ over the geodesic~$\gamma$.
Observe that if~$\gamma^{-1}$ denotes the pointwise inverse of~$\gamma$, we have $I(\rho,\gamma^{-1})=I(\rho,\gamma)=I(\rho,\gamma)^*$.

The next lemma provides the connection between representation theory and the Radon transform.

\begin{lemma}
\label{lma:fourier-radon}
If $f\in\C^G$, $\rho\in\hat G$ and $\gamma\in\Gamma_n$, we have $\widetilde{\rt f}(\rho,\gamma)=\tilde f(\rho)I(\rho,\gamma)$, where the Fourier transform is taken with respect to the first variable of~$\rt f$.
\end{lemma}

\begin{proof}
The proof is a mere calculation:
\begin{equation}
\begin{split}
\widetilde{\rt f}(\rho,\gamma)
&=
\sum_{x\in G}
\sum_{t\in C_n}
f(x\gamma(t))\rho(x)
\\&=
\sum_{y\in G}
f(y)\rho(y)
\sum_{t\in C_n}
\rho(\gamma^{-1}(t))
\\&=
\tilde f(\rho)I(\rho,\gamma).
\qedhere
\end{split}
\end{equation}
\end{proof}

Lemma~\ref{lma:fourier-radon} provides a new approach to the Radon transform.
We want to study whether the Radon transform is injective, that is, whether $\rt f=0$ implies $f=0$.
This problem has now been reduced to asking the following: if $\rho\in\hat G$ and~$A$ is a $\dim(\rho)\times\dim(\rho)$ matrix satisfying $AI(\rho,\gamma)=0$ for all nontrivial homomorphisms~$\gamma$, does~$A$ necessarily vanish?

To study this problem further, let us give the matrix~$I(\rho,\gamma)$ a description.

\begin{lemma}
\label{lma:I-proj}
Suppose $\gamma\in\Gamma_n$ and $\rho\in\hat G$.
Let $g\in S_n$ be a generating element and let~$V$ be the representation space of~$\rho$.
Then $n^{-1}I(\rho,\gamma):V\to V$ is the orthogonal projection to the subspace fixed by~$\rho(\gamma(g))$.
\end{lemma}

\begin{proof}
Since the matrix~$\rho(\gamma(g))$ is unitary, a unitary change of basis turns it into a diagonal matrix with~$n$th roots of unity as entries.
If~$z\in\C$ satisfies $z^n=1$, then $z+z^2+\dots+z^n$ is~$n$ if $z=1$ and zero otherwise.
Applying this to each entry in the diagonal matrix proves the claim.
\end{proof}

We have now found enough tools to give a second proof of theorem~\ref{thm:abelian} about the Radon transform on finite abelian groups.

\begin{proof}[{Second proof of theorem~\ref{thm:abelian}}]
Let~$G$ be a finite, abelian group of at least two elements.
All of its irreducible representations are one dimensional and we identify matrices with complex numbers.

If~$G$ is cyclic, there is a faithful irreducible representation~$\rho$.
Since $\rho(x)\neq1$ whenever $x\in G\setminus\{e\}$, lemma~\ref{lma:I-proj} implies that $I(\rho,\gamma)=0$ for all nontrivial homomorphisms $\gamma:C_n\to G$ for all~$n$.
If $f\in\C^G$ is such that $\tilde f(\rho)=1$ and $\tilde f(\sigma)=0$ for $\sigma\in\hat G\setminus\{\rho\}$, lemma~\ref{lma:fourier-radon} shows that $\widetilde{\rt f}=0$.
By lemma~\ref{lma:fourier-bij} this implies $\rt f=0$ although $f\neq0$, so the Radon transform is not injective.

Suppose then that~$G$ is not cyclic.
Then there is a cyclic group~$C_p$ so that~$G$ has a subgroup isomorphic to $C_p\times C_p$.
Let $\rho\in\hat G$.
Now $\sigma=\rho|_{C_p\times C_p}$ is a one dimensional representation of $C_p\times C_p$.
Every element in $C_p\times C_p$ has order~$1$ or~$p$, so every value of~$\sigma$ is a~$p$th root of unity.
There are only~$p$ such roots so~$\sigma$ cannot be faithful.
Therefore~$G$ has no faithful irreducible representations.


Now let $\rho\in\hat G$.
Since it is not faithful, there is $x\in G\setminus\{e\}$ so that $\rho(x)=1$.
Let~$n$ be the order of~$x$ and let $\gamma:C_n\to G$ be a homomorphism containing~$x$ in its image.
It follows from lemma~\ref{lma:I-proj} that $I(\rho,\gamma)=1$ and then from lemma~\ref{lma:fourier-radon} that $\rt f=0$ implies $f=0$.
\end{proof}

We can extract the following more general statements from the previous proof:
If~$G$ has a one dimensional faithful representation, the Radon transform cannot be injective.
But if no irreducible representation is faithful, the Radon transform is injective.

We can also give a representation theoretical characterization of the finite groups where the Radon transform is injective and characterize the kernel when it is not injective.
In order to do so, let us introduce some notation.

For $\rho\in\hat G$, let
\begin{equation}
k(\rho)=\bigcap_{p|\abs{G}}\bigcap_{\gamma\in\Gamma_p}\ker(I(\rho,\gamma))
\end{equation}
where the outer intersection is over all primes~$p$ that divide~$\abs{G}$.
Here $I(\rho,\gamma)$ is understood as a linear mapping $\C^{\dim(\rho)}\to\C^{\dim(\rho)}$ so~$k(\rho)$ is a subspace of~$\C^{\dim(\rho)}$.
Let~$K^*(\rho)$ denote the space of square matrices with the conjugate transpose of each row belonging to~$k(\rho)$.

The spaces~$k(\rho)$ and~$K^*(\rho)$ can be thought of as the vector and matrix kernels of the representation~$\rho$.
They describe the kernel of the Radon transform as we shall see next.

\begin{lemma}
\label{lma:I-K}
Let $\rho\in\hat G$.
A square matrix~$A$ satisfies $AI(\rho,\gamma)=0$ for all $\gamma\in\Gamma_p$ for all primes~$p$ if and only if $A\in K^*(\rho)$.
\end{lemma}

\begin{proof}
Recall that $I(\rho,\gamma)^*=I(\rho,\gamma)$.
The matrix~$A$ satisfies $AI(\rho,\gamma)=0$ if and only if each row~$a$ of~$A$ satisfies $I(\rho,\gamma)a^*=0$.
The claim follows now directly from the definition of~$K^*(\rho)$.
\end{proof}

For a representation $\rho:G\to GL(V)$, let
\begin{equation}
F_\rho=\{v\in V;\rho(x)v=v\text{ for some }x\in G\setminus\{e\}\}
\end{equation}
denote the set of fixed points.

\begin{lemma}
\label{lma:fixed-span}
Let~$G$ be a finite group and $\rho:G\to GL(V)$ an irreducible representation.
Then either~$V$ is spanned by~$F_\rho$ or $F_\rho=\{0\}$.
\end{lemma}

\begin{proof}
If $F_\rho=\{0\}$, then~$F_\rho$ clearly does not span~$V$.
We will show that the converse implication is also true.

If $v\in F_\rho$, then $\rho(x)v=v$ for some $x\in G\setminus\{e\}$.
Then $\rho(yxy^{-1})\rho(y)v=\rho(y)v$ for any $y\in G$ and so $\rho(y)v\in F_\rho$ for all $y\in G$.
Thus~$F_\rho$ is $G$\mbox{-}invariant, and so is its linear span.
Now if~$F_\rho$ does not span~$V$, its span has a nontrivial $G$\mbox{-}invariant complement by Maschke's theorem.
But since~$\rho$ is irreducible, this complement must be~$V$ itself, so $F_\rho=\{0\}$.
\end{proof}

Let us denote the set of irreducible representations of~$G$ without fixed points by
\begin{equation}
FPF(G)
=
\{\rho\in\hat G;F_\rho=\{0\}\}.
\end{equation}

\begin{lemma}
\label{lma:k-or-F}
Let~$G$ be a finite group and $\rho:G\to GL(V)$ a unitary irreducible representation.
Then either $F_\rho=\{0\}$ and $k(\rho)=V$ or~$F_\rho$ spans~$V$ and $k(\rho)=0$.
\end{lemma}

\begin{proof}
We will show that~$k(\rho)$ and the span~$S$ of~$F_\rho$ are orthogonal complements of each other in~$V$.
%

Suppose $v\in V$ satisfies $\rho(x)v=v$ for some nontrivial $x\in G$.
If we take a minimal geodesic~$\gamma$ passing through~$e$ and~$x$, it follows from lemma~\ref{lma:I-proj} that~$I(\rho,\gamma)$ is an orthogonal projection to a subspace of~$V$ containing~$v$.
Therefore every vector in~$k(\rho)$ is orthogonal to every $v\in S$.
On the other hand, if a vector $w\in V$ is orthogonal to~$S$, it is annihilated by every~$I(\rho,\gamma)$ and thus $w\in k(\rho)$.
This proves the desired orthogonality.
\end{proof}

\begin{theorem}
\label{thm:rep-radon-ker}
Let~$G$ be a finite group and $f:G\to\C$.
Then $\rt f=0$ if and only if $\tilde f(\rho)\in K^*(\rho)$ for all $\rho\in\hat G$, or, equivalently, $\tilde f(\rho)=0$ for all $\rho\in\hat G\setminus FPF(G)$.
\end{theorem}

\begin{proof}
By lemma~\ref{lma:prime} it suffices to consider geodesics of prime length.
By lemmas~\ref{lma:fourier-bij} and~\ref{lma:fourier-radon} we know that $\rt f=0$ if and only if $\tilde f(\rho)I(\rho,\gamma)=0$ for all $\rho\in\hat G$ and $\gamma\in\Gamma_p$ for all primes~$p$.
Invoking lemma~\ref{lma:I-K} then shows that $\rt f=0$ if and only if $\tilde f(\rho)\in K^*(\rho)$ for all $\rho\in\hat G$.
The last claimed equivalence follows from lemma~\ref{lma:k-or-F}.
\end{proof}


\begin{theorem}
\label{thm:rep-radon-inj}
Let~$G$ be a finite group.
The following are equivalent:
\begin{enumerate}
\item The Radon transform is injecitve on~$G$.
\item $k(\rho)=0$ for all $\rho\in\hat G$.
\item $FPF(G)=\emptyset$.
\item $G$ is not a Frobenius complement.
%
\end{enumerate}
\end{theorem}

\begin{proof}
Equivalence of the first three statements is provided by theorem~\ref{thm:rep-radon-ker} and lemma~\ref{lma:fourier-bij}.
Equivalence of the last two statements is a well-known characterization of Frobenius complements.
\end{proof}

\begin{remark}
\label{rmk:frobenius}
The mentioned characterization is a folklore result (see eg. the first paragraph of~\cite{S:frob-char}) in representation theory of finite groups.
The fact that every Frobenius complements have such fixed-point-free representations is given in~\cite[Theorem~18.1v]{P:perm-groups} but we have been unable to find a reference for the reverse direction in the literature.
For more information about Frobenius groups, we refer the reader to~\cite[Section~18]{P:perm-groups}, \cite[Section~V.8]{H:groups} and~\cite[Section~4.6]{RZ:profinite-groups}.
\end{remark}

Unfortunately, Frobenius complements are not easy to recognize.
There are some characterization results (see eg.~\cite{P:perm-groups,S:frob-char}) and the some of the results given in the present paper in Sections~\ref{sec:abelian}--\ref{sec:nonabelian} follow from known properties of Frobenius complements.
We point out that theorem~\ref{thm:rep-radon-inj} provides a new characterization of Frobenius complements in terms of the Radon transform.


\section{Generalizations and variations}
\label{sec:other}

\subsection{Remarks on infinite groups}
\label{sec:Z}

The situation in infinite groups is quite different.
To demonstrate this difference, let us consider the group~$\Z$.
In the definition of geodesics we should replace the finite cyclic groups with the infinite cyclic group~$\Z$.
To make sense of the Radon transform, we should restrict our attention to functions in~$\ell^1(\Z)$.
The Radon transform of a function in~$\ell^1(\Z)$ contains the sums over all cosets of infinite cyclic subgroups of~$\Z$.

The Radon transform on this space is injective.
Indeed, fix any $n\in\Z$ and consider the geodesic $\gamma_{n,m}:\Z\to\Z$, $\gamma_{n,m}(t)=n+mt$, for $m\geq1$.
If $f\in\ell^1(\Z)$, it follows from summability that
\begin{equation}
f(n)
=
\lim_{m\to\infty}\sum_{t\in\Z}f(\gamma_{n,m}(t)).
\end{equation}
Therefore the Radon transform is injective.

It follows from this and lemma~\ref{lma:subgroup} that if an infinite group~$G$ contains~$\Z$ as a subgroup (or, equivalently, has an element of infinite order), then the Radon transform is injective on~$\ell^1(G)$.
Restricting the Radon transform to periodic geodesics removes this phenomenon, but it also removes all geodesics from~$\Z$.

For an example of different kind, consider the (additive) quotient group $G=\Q/\Z$.
Every element has finite order, so the Radon transform is reasonable for any function $f:G\to\C$.
Now it is an elementary calculation to observe that the function $f(x+\Z)=e^{2\pi ix}$ has zero Radon transform.
It follows from this noninjectivity result that all finite subgroups of~$G$ are cyclic; for if there was a noncyclic finite subgroup, the Radon transform would be injective by lemma~\ref{lma:subgroup} and theorem~\ref{thm:abelian}.

We conclude our discussion of infinite groups with the following open problem.
We have observed that for an infinite group~$G$ where every element has finite order, the Radon transform is injective if there is a finite subgroup on which it is injective.
Is the converse true?
Since all elements have finite order, the functions $f:G\to\C$ need not satisfy any summability condition.


\subsection{Discrete flows}
\label{sec:flow}

One can also generalize geodesic flows to discrete settings without group structure.
Let~$X$ be a set.
A curve on~$X$ is a mapping from~$\Z$ (or an integer interval) to~$X$.
Let $s:X\times X\to X$ be a function.
We say that a curve is an $s$\mbox{-}geodesic if for any two adjacent points $a,b\in X$ on a curve, the next point after~$a$ and~$b$ is~$s(a,b)$.

To make this resemble geodesic flow, we want the function~$s$ to satisfy the following: $s(a,b)=b\iff a=b$ and $s(a,b)=c\iff s(c,b)=a$.
These conditions ensure that the reverse of an $s$\mbox{-}geodesic is also an $s$\mbox{-}geodesic and that a nonconstant $s$\mbox{-}geodesic never terminates.

On a group we set $s(a,b)=ba^{-1}b$.
This satisfies the above requirements and gives rise to geodesics as discussed in this article.
If~$X$ is finite, every geodesic is periodic also without group structure.

We can now define the Radon transform corresponding to the successor function~$s$ by replacing the geodesics defined in section~\ref{sec:def-geod} by $s$\mbox{-}geodesics.

A flow satisfying the above requirements need not come from a group structure.
On any set~$X$ one can define a flow by $s(a,b)=a$; all nonconstant geodesics have length two.
If~$X$ were a group, a geodesic through the identity would be a cyclic subgroup and thus~$\abs{X}$ would have to be even, although the flow can be defined for any~$\abs{X}$.

The Radon transform corresponding to this particular flow is injective on any set~$X$ with at least three elements.
To prove this, observe that the Radon transform~$\rt f$ of a function $f:X\to\C$ gives the sum of~$f$ over any pair of points in~$X$.

This flow can also been seen as a discrete time dynamical system on $X\times X$ taking~$(a,b)$ to~$(b,s(a,b))$.
The diagonal $\Delta\subset X\times X$ is precisely the set of stationary points.
Projections of nontrivial orbits of this system to the first (or second) component are precisely the $s$\mbox{-}geodesics in the space~$X$.

\subsection{A variant of the Radon transform}
\label{sec:var}

Above we defined geodesics as cosets of nontrivial cyclic subgroups.
Now we change this definition and require geodesics to be additionally maximal\footnote{This definition was suggested by Peter Michor.}.
A cyclic subgroup is maximal if it is not contained in a larger cyclic subgroup.
We call these new geodesics maximal geodesics and the corresponding transform the maximal Radon transform and denote it by~$\mrt$ instead of~$\rt$.
Note that by lemma~\ref{lma:prime} it suffices to analyze the \emph{minimal} nontrivial geodesics to study the nonmaximal Radon transform.

This definition causes several changes to the theory.
There are (typically strictly) fewer maximal geodesics than geodesics. Thus if~$\mrt$ is injective, so is~$\rt$.
Also, not all subgroups are totally geodesic (w.r.t. maximal geodesics). Totally geodesic subgroups correspond to Lie subgroups. Positive dimensional Lie groups can have both Lie subgroups and discrete subgroups, and this structure has now its discrete counterpart.
Many of the results change, and we illustrate these changes in this section without developing a full theory of the maximal Radon transform.

In the case of Lie groups (with geodesics parameterized by~$S^1$), every maximal abelian subgroup is totally geodesic.
This is also true in the case of maximal discrete geodesics if the subgroup is cyclic or the whole group is abelian, but not necessarily otherwise.
Consider for example the symmetric group~$S_5$ and the subgroup~$H$ generated by~$(1,2)$ and~$(3,4)$.
This is a maximal abelian subgroup, but the cyclic subgroup generated by~$(1,2)$ is maximal in~$H$ but not in~$S_5$.
The subgroup of~$S_5$ generated by~$(1,2)$ is contained in the one generated by $(1,2)(3,4,5)$.

\begin{proposition}
\label{prop:max-square-cyclic}
For any prime $p>1$, the maximal Radon transform is injective in $C_p\times C_p$.
On the other hand, the maximal Radon transform is not injective on any cyclic group.
\end{proposition}

\begin{proof}
The proof of the first claim is the same as that of lemma~\ref{lma:square} since all geodesics are maximal.

The second one follows from the fact that if~$\mrt$ were injective,~$\rt$ would also be, contradicting lemma~\ref{lma:cyclic}.
Another way to see this is that there is only one maximal cyclic subgroup, namely the group itself, so any function with zero average is in the kernel of the maximal Radon transform.
\end{proof}

Recall that if~$G$ is a finite group and~$p$ is a prime, then a Sylow $p$\mbox{-}subgroup of~$G$ is a maximal subgroup of~$G$ such that the order of each element is a power of~$p$.
A Sylow subgroup is a Sylow $p$\mbox{-}subgroup for some prime~$p$.

\begin{proposition}
\label{prop:max-cyclic-factor}
Consider the finite group $G=G_1\times G_2$ where~$G_1$ is cyclic and~$\abs{G_1}$ and~$\abs{G_2}$ are coprime.
The maximal Radon transform is not injective on~$G$.

In particular, if~$G$ is a finite abelian group with a cyclic Sylow subgroup, the maximal Radon transform is not injective on it.
\end{proposition}

\begin{proof}
Let $\pi_i:G\to G_i$, $i=1,2$, be the two natural projections.
We will first show that if~$C$ is a maximal cyclic subgroup of~$G$, then $\pi_1(C)=G_1$.

Suppose~$C$ is a cyclic subgroup of~$G$ so that $\pi_1(C)\neq G_1$.
We will show that it is not maximal.
If~$g$ is a generator of~$C$ and~$g_1$ a generator of~$G_1$, we know that $g_1\notin\pi_1(C)$.
But~$g_1$ generates~$G_1$, so there is an integer~$m$ so that $g_1^m=\pi_1(g)$.
Let~$n_1$ and~$n_2$ be the orders of~$g_1$ and~$\pi_2(g)$, respectively.
Since~$n_1$ and~$n_2$ are coprime by assumption, there is an integer~$c$ so that $c\equiv m\pmod{n_1}$ and $c\equiv1\pmod{n_2}$.
If we let $h=(g_1,\pi_2(g))$, we have $h^m=g$.
But~$h$ is not contained in the cylcic subgroup~$C$ generated by~$g$, so~$C$ is not maximal.

There is a nonzero function $f:G_1\to\C$ with zero average (this is essentially the second claim of proposition~\ref{prop:max-square-cyclic}).
Let us define $g:G\to\C$ by $g=f\circ\pi_1$.
Since maximal geodesics have surjective projections to~$G_1$, the sum of~$g$ over every maximal geodesic is a multiple of the sum of~$f$ over~$G_1$ and thus zero.
Therefore~$g$ is a nonzero function but $\mrt g=0$.

The second claim follows from Kronecker's decomposition theorem and the first claim.
\end{proof}

\begin{proposition}
\label{prop:max-cyclic-subgroups}
Let~$G$ be a finite group and let~$S$ denote the set of its maximal cyclic subgroups.
Then the maximal Radon transform is not injective on~$G$ if
\begin{equation}
\sum_{H\in S}\abs{H}^{-1}
<
1+\frac{\abs{S}-1}{\abs{G}}.
\end{equation}
\end{proposition}

\begin{proof}
The proof is essentially the same as that of proposition~\ref{prop:cyclic-subgroups}.
\end{proof}

The following proposition is a discrete analogue of~\cite[Proposition~2.4]{I:lie-radon}.
When compared to the case of Lie groups,~$H$ should be understood as a discrete subgroup.

\begin{proposition}
\label{prop:quotient}
Let~$G$ be a finite group and~$H$ such a subgroup that no cyclic subgroup of~$H$ is a maximal subgroup of~$G$.
If the maximal Radon transform is injective on~$G$, then the (nonmaximal) Radon transform is injective on the quotient~$G/H$.
\end{proposition}

Before embarking on the proof, let us take a moment to put the statement in some context since the situation is quite delicate.
The analogous result is true for a Lie group~$G$ and a discrete normal subgroup~$H$ where both Radon transforms are the usual Radon transform on Lie groups~\cite[Proposition~2.4]{I:lie-radon}.
Proposition~\ref{prop:quotient} is not symmetric: we assume injectivity of~$\mrt^G$ (which is a stronger condition than injectivity of~$\rt^G$) and we obtain injectivity of~$\rt^{G/H}$ (weaker than injectivity of~$\mrt^{G/H}$).
We have again indicated the underlying group by a superscript in the transform.
As discussed in the beginning of section~\ref{sec:nonabelian}, the statement is false if we consider~$\rt^G$ and~$\rt^{G/H}$.

The statement is also false if we consider~$\mrt^G$ and~$\mrt^{G/H}$.
Consider for example the group $G=C_6\times C_6$ on which maximal Radon transform is injective\footnote{We do not have an elegant proof for this fact, but it is easy to check computationally. There are 12 maximal cyclic subgroups and each has 6 distinct cosets. A function $f:G\to\C$ can be regarded as a vector in~$\C^{36}$ and the question is whether the 72 conditions (sums over geodesics) determine it uniquely. The rank of the corresponding $36\times72$ matrix turns out to be 36.}.
Take any subgroup~$H$ with two elements; it will satisfy the assumption of proposition~\ref{prop:quotient} because all maximal geodesics have length six.
Now the maximal Radon transform is not injective on $G/H\approx C_3\times C_6$ by proposition~\ref{prop:max-cyclic-factor}.

\begin{proof}[{Proof of proposition~\ref{prop:quotient}}]
Suppose $f:G/H\to\C$ satisfies $\rt^{G/H}f=0$.
If $\pi:G\to G/H$ is the natural projection, we want to show that $\pi^*f:G\to\C$ (and thus~$f$ itself) vanishes by considering its maximal Radon transform.
To this end, let $\gamma:C_n\to G$ be any homomorphism whose image is a maximal cyclic subgroup and let $x\in G$.
Then
\begin{equation}
\mrt^G(\pi^*f)(x,\gamma)
=
\rt^G(\pi^*f)(x,\gamma)
=
\rt^{G/H}f(\pi(x),\pi\circ\gamma).
\end{equation}
Suppose we know that $\pi\circ\gamma:C_n\to G/H$ is nontrivial.
Then we get $\pi^*f=0$ using the assumption $\rt^{G/H}f=0$ and injectivity of~$\mrt^G$.

Let us then show that $\pi\circ\gamma$ is indeed nontrivial.
If~$g$ is a generator of~$C_n$, we want to show that $\pi(\gamma(g))\neq e$.
If this were not the case, we would have $\gamma(g)\in H$, whence $\gamma(C_n)\subset H$.
We have assumed that~$\gamma(C_n)$ is a maximal cyclic subgroup of~$G$ and that~$H$ cannot contain such a subgroup.
This contradiction concludes the proof.
\end{proof}

\section*{Acknowledgements}

This work was supported by the Academy of Finland (Centre of Excellence in Inverse Problems Research) and an ERC Starting Grant.
The author wishes to express his gratitude to the help received at MathOverflow.
The folklore characterization of Frobenius complements via representation theory would never have made it here without discussions with Frieder Ladisch.

%
%
%
%


%

\bibliographystyle{abbrv}
\bibliography{ip}

\end{document}